%% file: paper_v06_arXiv.tex
\documentclass[11pt]{article}
\usepackage{amsfonts, amssymb, graphicx, amsmath}
\usepackage{fourier}
\usepackage{natbib}
\usepackage{fullpage}
\usepackage[usenames]{color}
\usepackage[pdftex]{hyperref}
\hypersetup{
    colorlinks=true,       
    linkcolor=blue,        
    citecolor=blue,        
    filecolor=blue,        
    urlcolor=blue          
}
\RequirePackage{natbib}
\usepackage{fourier}
\usepackage{amsmath,amssymb,latexsym,amsthm}
\usepackage{graphicx}
\usepackage[ruled,vlined]{algorithm2e}

\include{header}

\newcommand*\samethanks[1][\value{footnote}]{\footnotemark[#1]}
\newcommand{\uu}[1]{\underline{#1}}
\newcommand{\bb}[1]{\bar{#1}}
\newcommand{\HH}{g}
\begin{document}
\title{Escaping the Local Minima via Simulated Annealing: Optimization of Approximately Convex Functions}
\date{}
\author{Alexandre Belloni  \thanks{The Fuqua School of Business, Duke University} \and Tengyuan Liang \thanks{Department of Statistics, The Wharton School, University of Pennsylvania} \and
Hariharan Narayanan \thanks{Department of Statistics and Department of Mathematics, University of Washington} \and Alexander Rakhlin \samethanks[2]}
\maketitle

\begin{abstract} We consider the problem of optimizing an approximately convex function over a bounded convex set in $\mathbb{R}^n$  using only function evaluations. The problem is reduced to sampling from an \emph{approximately} log-concave distribution using the Hit-and-Run method, which is shown to have the same $\mathcal{O}^*$ complexity as sampling from log-concave distributions. In addition to extend the analysis for log-concave distributions to approximate log-concave distributions, the implementation of the 1-dimensional sampler of the Hit-and-Run walk requires new methods and analysis. The algorithm then is based on simulated annealing which does not relies on first order conditions which makes it essentially immune to local minima.

We then apply the method to different motivating problems. In the context of zeroth order stochastic convex optimization, the proposed method produces an $\epsilon$-minimizer after $\mathcal{O}^*(n^{7.5}\epsilon^{-2})$ noisy function evaluations  by inducing a $\mathcal{O}(\epsilon/n)$-approximately log concave distribution. We also consider in detail the case when the ``amount of non-convexity'' decays towards the optimum of the function. Other applications of the method discussed in this work include private computation of empirical risk minimizers, two-stage stochastic programming, and approximate dynamic programming for online learning.
\end{abstract}


\section{Introduction and Problem Formulation}

Let $\K\subset \reals^n$ be a convex set, and let $F:\reals^n\to\reals$ be an  approximately convex function over $\K$ in the sense that
\begin{align}
	\label{eq:linfty_approx}
	\sup_{x\in\K}|F(x)-f(x)|\leq \epsilon/n
\end{align}
for some convex function $f:\reals^n\to\reals$ and $\epsilon>0$. In particular, $F$ may be discontinuous. We seek to find $x\in\K$ such that
\begin{align}
	\label{eq:opt_objective}
F(x) - \min_{y \in \K} F(y) \leq \epsilon
\end{align}
using only function evaluations of $F$. This paper presents a randomized method based on simulated annealing that satisfies \eqref{eq:opt_objective} in expectation (or with high probability). Moreover, the number of required function evaluations of $F$ is at most $\mathcal{O}^*(n^{4.5})$ (see Corollary~\ref{cor:4.5}), where $\mathcal{O}^*$ hides polylogarithmic factors in $n$ and $\epsilon^{-1}$. Our method requires only a membership oracle for the set $\K$. In Section \ref{sec:de-fluc}, we consider the case when the amount of non-convexity in \eqref{eq:linfty_approx} can be much larger than $\epsilon/n$ for points away from the optimum.

In the oracle model of computation, access to function values at queried points is referred to as the zeroth-order information.
Exact function evaluation of $F$ may be equivalently viewed as approximate function evaluation of the convex function $f$, with the oracle returning a value
\begin{align}
	\label{eq:approx_oracle}
	F(x) \in [f(x)-\epsilon/n, f(x)+\epsilon/n].
\end{align}

A closely related problem is that of convex optimization with a \emph{stochastic} zeroth order oracle. Here, the oracle returns a noisy function value $f(x)+\eta$. If $\eta$ is zero-mean and subgaussian, the function values can be averaged to emulate, with high probability, the approximate oracle \eqref{eq:approx_oracle}. The randomized method we propose has an $\mathcal{O}^*(n^{7.5}\epsilon^{-2})$ oracle complexity for convex stochastic zeroth order optimization, which, to the best of our knowledge, is the best that is known for this problem. We refer to Section~\ref{sec:stoch_zeroth} for more details.

The motivation for studying zeroth-order optimization is plentiful, and we refer the reader to \cite{conn2009introduction} for a discussion of problems where derivative-free methods are essential. In Section~\ref{sec:apps} we sketch three areas where the algorithm of this paper can be readily applied: private computation with distributed data, two-stage stochastic programming, and online learning algorithms.


\section{Prior Work}

The present paper rests firmly on the long string of work by Kannan, Lov\'asz, Vempala, and others \citep{lovasz1993random,kannan1997random,kalai2006simulated,lovasz2006fast,lovasz2006hit,lovasz2007geometry}. In particular, we invoke the key lower bound on conductance of Hit-and-Run from \cite{lovasz2006fast} and use the simulated annealing technique of \cite{kalai2006simulated}. Our analysis extends Hit-and-Run to approximately log-concave distributions which required new theoretical results and implementation adjustments. In particular, we propose a unidimensional sampling scheme that mixes fast to a truncated approximately log-concave distribution on the line.

Sampling from $\beta$-log-concave distributions was already studied in the early work of \cite{applegate1991sampling} with a discrete random walk based on a discretization of the space.  In the case of non-smooth densities and unrestricted support, sampling from approximate log-concave distributions has also been studied in \cite{belloni2009computational} where the hidden convex function $f$ is quadratic. This additional structure was motivated by the central limit theorem in statistical applications and leads to faster mixing rates. Both works used ball walk-like strategies. Neither work considered random walks that allow for long steps like Hit-and-Run.

The present work was motivated by the question of information-based complexity of zeroth-order stochastic optimization. The paper of \cite{AgaFosHsuKakRak13siam} studies a somewhat harder problem of regret minimization with zeroth-order feedback. Their method is based on the pyramid construction of \cite{NemYud83} and requires $\mathcal{O}(n^{33}\epsilon^{-2})$ noisy function evaluations to achieve a regret (and, hence, an optimization guarantee) of $\epsilon$. The method of \cite{liang2014zeroth} improved the dependence on the dimension to $\mathcal{O}^*(n^{14})$ using a Ball Walk on the epigraph of the function in the spirit of \cite{bertsimas2004solving}.  The present paper further reduces this dependence to $\mathcal{O}^*(n^{7.5})$ and still achieves the optimal $\epsilon^{-2}$ dependence on the accuracy. The best known lower bound for the problem is $\Omega(n^2\epsilon^{-2})$ (see \cite{Shamir12}).

Other relevant work includes the recent paper of \cite{dyer2013simple} where the authors proposed a simple random walk method that requires only  approximate function evaluations. As the authors mention, their algorithm only works for smooth functions  and sets $\K$ with smooth boundaries --- assumptions that we would like to avoid. Furthermore, the effective dependence of \cite{dyer2013simple} on accuracy is worse than $\epsilon^{-2}$.

\section{Preliminaries}

Throughout the paper, the functions $F$ and $f$ satisfy \eqref{eq:linfty_approx} and $f$ is convex. The Lipschitz constant of $f$ with respect to $\ell_\infty$ norm will be denoted by $L$, defined as the smallest number such that $|f(x) - f(y)| \leq L \|x - y \|_{\infty}$ for $x, y \in \K$. Assume the convex body $\K\subseteq \reals^n$ to be well-rounded in the sense that there exist $r,R>0$ such that $\B_2^n(r) \subseteq \K \subseteq  \B_2^n(R)$ and $R/r \leq \mathcal{O}(\sqrt{n})$.\footnote{This condition can be relaxed by applying a pencil construction as in \cite{lovasz2007geometry}.} For a non-negative function $g$, denote by $\pi_g$ the normalized probability measure induced by $g$ and supported on $\K$.

\begin{definition}
	A function $h:\K\to\reals_+$ is log-concave if
	$$h(\alpha x + (1-\alpha)y) \geq h(x)^\alpha h(y)^{1-\alpha}$$
	for all $x,y\in \K$ and $\alpha\in [0,1]$. A function is called $\beta$-log-concave for some $\beta\geq 0$ if
	$$h(\alpha x + (1-\alpha)y) \geq e^{-\beta} h(x)^\alpha h(y)^{1-\alpha}$$
	for all $x,y\in \K$ and $\alpha\in [0,1]$.
\end{definition}

\begin{definition}
A function $g:\K\to\reals_+$ is $\xi$-approximately  log-concave if there is a log-concave function $h:\K\to\reals_+$ such that
 $$\sup_{x\in \K} |\log h(x) - \log g(x)| \leq \xi.$$
\end{definition}

\begin{lemma}
	\label{lem:linfty_beta_concave}
	If the function $g$ is $\beta/2$-approximately  log-concave, then $g$ is $\beta$-log-concave.
\end{lemma}

For one-dimensional functions, the above lemma can be reversed:
\begin{lemma}[\cite{belloni2009computational}, Lemma 9]
	\label{lem:1dsandwich}
If $g$ is a unidimensional $\beta$-log-concave function, then there exists a log-concave function $h$ such that
\begin{align*}
e^{-\beta} h(x) \leq g(x) \leq h(x) \ \ \mbox{for all} \ x \in \reals.
\end{align*}
\end{lemma}

\begin{remark}[Gap Between $\beta$-Log-Concave Functions and $\xi$-Approximate Log-Concave Functions]
A consequence of Lemma \ref{lem:1dsandwich} is that $\beta$-log-concave functions are equivalent to $\beta$-approximately  log-concave functions when the domain is unidimensional. However, such equivalence no longer holds in higher dimensions.
In the case the domain is $\reals^n$, \cite{green1952approximately,cholewa1984remarks} established that $\beta$-log-concave functions are $\frac{\beta}{2}\log_2 (2n)$-approximately  log-concave. \cite{laczkovich1999local} showed that there are functions such that the factor that relates these approximations cannot be less than $\frac{1}{4}\log_2 (n/2)$. \qed
\end{remark}

We end this section with two useful lemmas that can be found in \cite{lovasz2007geometry}.
\begin{lemma}[\cite{lovasz2007geometry}, Lemma 5.19]
	\label{lem:volume}
	Let $h: \mathbb{R}^n\rightarrow \mathbb{R}$ be a log-concave function. Define $M_h := \max h$ and $L_h(t) = \{ x\in \mathbb{R}^n: h(x)\geq t \}$. Then for $0<s<t<M_h$
	\begin{align*}
	\frac{{\sf vol}(L_h(s))}{{\sf vol}(L_h(t))} \leq \left( \frac{\log(M_h/s)}{\log(M_h/t)}  \right)^n .
	\end{align*}
\end{lemma}

\begin{lemma}[\cite{lovasz2007geometry}, Lemma 5.6(a)]
	\label{lem:tails}
Let $X$ be a random point drawn from a log-concave distribution $h: \mathbb{R} \rightarrow \mathbb{R}_+$ and let $M_h := \max_{x\in \mathbb{R}} h(x)$. Then for every $t \geq 0$
\begin{align*}
\mathbb{P}(h(X) \leq t) < \frac{t}{M_h}.
\end{align*}
\end{lemma}

\section{Sampling from Approximate Log-Concave Distributions via Hit-and-Run}

In this section we analyze the Hit-and-Run procedure to simulate random variables from a distribution induced by an approximate log-concave function. The Hit-and-Run algorithm is as follows.

\vspace{1cm}
\begin{algorithm}[H]
\SetAlgoLined
\label{alg:hit_run}
 \KwIn{a target distribution $\pi_g$ on $\K$ induced by a nonnegative function $g$; $x\in{\sf dom}(g)$; linear transformation $\Sigma$; number of steps $m$}
 \KwOut{a point $x' \in {\sf dom}(g)$ generated by one-step Hit-and-Run walk}
 initialization: a starting point $x \in {\sf dom}(g)$ \;
 \For{$i=1,\ldots,m$}{
 1. Choose a random line $\ell$ that passes through $x$. The direction is uniform from the surface of ellipse given by $\Sigma$ acting on sphere \;
 2. On the line $\ell$ run the unidimensional rejection sampler with $\pi_g$ restricted to the line (and supported on $\K$) to propose a successful next step $x'$ \;
 }
 \caption{Hit-and-Run}
\end{algorithm}
\vspace{0.5cm}

In order to handle approximate log-concave functions we need to address implementation issues and address the theoretical difficulties caused by deviations from log-concavity which can include discontinuities. The main implementation difference lies is the unidimensional sampler. No longer a binary search yields the maximum over the line and its end points since $\beta$-log-concave functions can be discontinuous and multimodal.
We now turn to these questions.

\subsection{Unidimensional sampling scheme}

As a building block of the randomized method for solving the optimization problem \eqref{eq:opt_objective}, we introduce a one-dimensional sampling procedure. Let $g$ be a unidimensional $\beta$-log-concave function on a bounded line segment $\ell$, and let $\pi_g$ be the induced normalized measure. The following guarantee will be proved in this section.
\begin{lemma}
	\label{lem:1dsampler}
Let $g$ be a $\beta$-log-concave function and let $\ell$ be a bounded line segment $\ell$ on $\K$. Given a target accuracy $\tilde{\epsilon}\in(0, e^{-2\beta}/2)$, Algorithm~\ref{alg:rej_samp} produces a point $X\in \ell$ with a distribution $\hat \pi_{g,\ell}$ such that
$$d_{\sf tv}(\pi_{g,\ell},\hat\pi_{g,\ell})\leq 3e^{2\beta}\tilde{\epsilon}.$$
Moreover, the method requires $\mathcal{O}^*(1)$ evaluations of the unidimensional $\beta$-log-concave function $g$ if $\beta$ is $\mathcal{O}(1)$.
\end{lemma}

\vspace{1cm}
\begin{algorithm}[H]
\SetAlgoLined
\label{alg:rej_samp}
 \KwIn{unidimensional $\beta$-log-concave function $g$ defined on a bounded segment $\ell=[\uu{x},\bb{x}]$; accuracy $\tilde{\epsilon}>0$}
 \KwOut{A sample $x$ with distribution $\hat\pi_{g,\ell}$ close to $\pi_{g,\ell}$}
 Initialization: \textbf{(a)} compute a point $p \in \ell$ s.t.  $g(p) \geq e^{-3\beta} \max_{z \in \ell} g(z)$ \;
 \textbf{(b)} given target accuracy $\tilde{\epsilon}$, find two points $e_{-1},e_1$ on two sides of $p$ s.t.
 	\begin{equation}\label{eq:init_ei}
 \begin{array}{c}
e_{-1}=\uu{x} \ \ \mbox{if }\ \   g(\uu{x})\geq \frac{1}{2}e^{-\beta} \tilde{\epsilon} g(p), \ \ \ \frac{1}{2}e^{-\beta} \tilde{\epsilon} g(p) \leq g(e_{-1}) \leq \tilde{\epsilon} g(p) \ \ \mbox{otherwise}\\
\\
e_{1}=\bb{x} \ \ \mbox{if }\ \   g(\bb{x})\geq \frac{1}{2}e^{-\beta} \tilde{\epsilon} g(p), \ \ \	\frac{1}{2}e^{-\beta} \tilde{\epsilon} g(p) \leq g(e_1) \leq \tilde{\epsilon} g(p) \ \ \mbox{otherwise};
\end{array}	\end{equation}
 \While{sample rejected}{
  pick $x \sim {\sf unif}([e_{-1},e_1])$ and pick $r \sim {\sf unif}([0,1])$ independently\;
  \eIf{$r  \leq g(x)/\{g(p)e^{3\beta}\}$}{
   accept $x$ and stop\;
   }{
   reject $x$ \;
  }
 }
 \caption{Unidimensional rejection sampler}
\end{algorithm}
\vspace{0.5cm}

The proposed method for sampling from the $\beta$-log-concave function $g$ is a  \emph{rejection} sampler that requires two initialization steps. We first show how to implement step $\mathbf{(a)}$.

\begin{algorithm}[H]
\SetAlgoLined
\label{alg:initial}
 \KwIn{unidimensional $\beta$-log-concave function $g$ defined on a bounded interval $\ell = [\uu{x},\bb{x}]$}
 \KwOut{a point $p \in \ell$ s.t.  $g(p) \geq e^{-3\beta} \max_{z \in \ell} g(z)$}
 \While{did not stop}{
  set $x_l = \frac{3}{4} \uu{x} + \frac{1}{4} \bb{x}$, $x_c = \frac{1}{2} \uu{x} + \frac{1}{2} \bb{x}$ and $x_r = \frac{1}{4} \uu{x} + \frac{3}{4} \bb{x}$ \;
  \eIf{$|\log g(x_l) - \log g(x_r)|>\beta $}{
   set $[\uu{x},\bb{x}]$ as either $[x_l, \bb{x}]$ or $[\uu{x}, x_r]$ accordingly \; 
   }{
       \uIf{$|\log g(x_l) - \log g(x_c)|>\beta$}{
            set $[\uu{x},\bb{x}]$ as either $[x_l, \bb{x}]$ or $[\uu{x}, x_c]$ accordingly \; 
          }
       \uElseIf{$|\log g(x_r) - \log g(x_c)|>\beta$}{
            set $[\uu{x},\bb{x}]$ as either $[\uu{x}, x_r]$ or $[x_c, \bb{x}]$ accordingly \; 
          }
        \Else{ output $p = {\displaystyle \arg\max_{x\in\{x_l,x_c,x_r\}}}  f(x)$ and stop \;}
     }
 }
 \caption{Initialization Step \textbf{(a)}}
\end{algorithm}
\vspace{0.5cm}

For the $\beta$-log-concave function $g$, let $h$ be a log-concave function  in Lemma~\ref{lem:1dsandwich} and let $\tilde{L}$ denote the Lipschitz constant of the convex function ~$-\log h$. In the following two results, the $\mathcal{O}^*$ notation hides a $\log(\tilde{L})$ factor.

\begin{lemma}[Initialization Step \textbf{(a)}]\label{Lemma:InitStepa}
Algorithm~\ref{alg:initial} finds a point $p \in \ell$ that satisfies $g(p) \geq e^{-3\beta} \max_{z \in \ell} g(z)$.
Moreover, this step requires $\mathcal{O}^*(1)$ function evaluations.
\end{lemma}

\begin{lemma}[Initialization Step \textbf{(b)}]\label{Lemma:InitUni}
Let $\ell=[\uu{x},\bb{x}]$ and $p\in \ell$. The binary search algorithm finds $e_{-1}\in [\uu{x},p]$ and $e_1 \in [ p, \bb{x}]$ such that \eqref{eq:init_ei} holds.
Moreover, this step requires $\mathcal{O}^*(1)$ function evaluations.
\end{lemma}

According to Lemmas~\ref{lem:1dsampler},~\ref{Lemma:InitStepa},~\ref{Lemma:InitUni},  the unidimensional sampling method produces a sample from a distribution that is close to the desired $\beta$-log-concave distribution. Furthermore, the method requires a number of queries that is logarithmic in all the parameters.

\subsection{Mixing time}

 In this section, we will analyze mixing time of the Hit-and-Run algorithm with a $\beta/2$-approximate log-concave function $\HH$, namely
\begin{align}
	\label{eq:near_log_concave}
	\exists ~~\text{ log-concave }~~~ h ~~~\text{ s.t. }~~~~~ \sup_{\K}| \log \HH - \log h|\leq \beta/2.
\end{align}
In particular, this implies that $\HH$ is $\beta$-log-concave, according to Lemma~\ref{lem:linfty_beta_concave}. In this section, we provide the analysis of Hit-and-Run with the linear transformation $\Sigma=I$ and remark that the results extend to other linear transformations employed to round the log-concave distributions.

The mixing time of a geometric random walk can be bounded through the spectral gap of the induced Markov chain. In turn, the spectral gap relates to the so called conductance which has been a key quantity in the literature. Consider the transition probability of Hit-and-Run with a density $g$, namely
$$
P_u^{\HH}(A) =  \frac{2}{n \pi_n} \int_A \frac{\HH(x) dx}{\mu_\HH (u,x) |x - u|^{n-1}}
$$
where
$\mu_\HH(u,x) = \int_{\ell(u,x) \cap \K} \HH(y) dy$. Let $\pi_\HH(x) = \frac{\HH(x)}{\int_{y \in \K} \HH(y) dy}$ be the probability measure induced by the function $\HH$. The conductance for a set $S\subset \K$ with $0<\pi_\HH(S)<1$ is defined as
$$
\phi^\HH(S)  = \frac{\int_{x \in S} P_x^\HH(\K \backslash S) d \pi_\HH}{ \min\{ \pi_\HH(S), \pi_\HH(\K\backslash S) \}},
$$
and $\phi^\HH$ is the minimum conductance over all measurable sets. The $s$-conductance is, in turn, defined as
$$
\phi_s^\HH  = \inf_{S\subset \K, s<\pi_\HH(S)\leq 1/2} \frac{\int_{x \in S} P_x^\HH(\K \backslash S) d \pi_\HH}{ \pi_\HH(S) -s }.
$$ By definition we have $\phi^\HH \leq \phi^\HH_s$ for all $s>0$.

The following theorem provides us an upper bound on the mixing time of the Markov chain based on conductance. Let $\sigma^{(0)}$ be the initial distribution and $\sigma^{(m)}$ the distribution of the $m$-th step of the random walk of Hit-and-Run with exact sampling from the distribution $\pi_\HH$ restricted to the line.

\begin{theorem}[\cite{lovasz1993random}; \cite{lovasz2007geometry}, Lemma 9.1]
	\label{thm:contraction_LV}
	Let $0 < s \leq 1/2$ and let $g:\K\to\reals_+$ be arbitrary. Then for every $m \geq 0$,
$$
d_{\sf tv}(\pi_g,\sigma^{(m)}) \leq H_0 \left(1 - \frac{(\phi^g)^2}{2}\right)^m\ \ \ \mbox{and} \ \ \ d_{\sf tv}(\pi_g,\sigma^{(m)}) \leq H_s + \frac{H_s}{s} \left(1 - \frac{(\phi_s^g)^2}{2}\right)^m
$$
where
$$
H_0=\sup_{x\in \K} \pi_g(x)/\sigma^{0}(x) \ \ \ \mbox{and} \ \ \ H_s = \sup\{|\pi_g(A) - \sigma^{(0)}(A) |: \pi_g(A) \leq s \}.
$$
\end{theorem}

Building on \cite{lovasz2006fast}, we prove the following result that provides us with a lower bound on the conductance for Hit-and-Run induced by a log-concave $h$. The proof of the result below follows the proof of Theorem 3.7 in \cite{lovasz2006fast} with modifications to allow unbounded sets $\K$ without truncating the random walk. 
\begin{theorem}[Conductance Lower Bound for Log-concave Measures with Unbounded Support]
	\label{thm:cond_lower_bound}
Let $h$ be a log-concave function in $\mathbb{R}^n$ such that the level set of measure $\frac{1}{8}$ contains a ball of radius $r$. Define $R =  (\E_h \| X - z_h  \|^2 )^{1/2}$, where $z_h = \E_h X$ and $X$ is sampled from the log concave measure induce by $h$. Then for any subset $S$, with $\pi_h(S) = p \leq \frac{1}{2}$, the conductance of Hit-and-Run satisfies
$$
\phi^h(S) \geq \frac{1}{C_1 n \frac{R}{r} \log^2 \frac{nR}{rp}}
$$
where $C_1>0$ is a universal constants.
\end{theorem}

Although Theorem \ref{thm:cond_lower_bound} is new, very similar conductance bounds allowing for unbounded sets were establish before. Indeed in Section 3.3 of  \cite{lovasz2006fast} the authors discuss the case of unbounded $\K$ and propose to truncate the set to its effective diameter and use the fact that this distribution would be close to the distribution of the unrestricted set. Such truncation needs to be enforces which requires to change the implementation of the algorithm and lead to another (small) layer of approximation errors. Theorem \ref{thm:cond_lower_bound}  avoids this explicit truncation and truncation is done implicitly in the proof only. We note that when applying the simulated annealing technique, even if we start with a bounded set, by diminishing the temperature, we are effectively stretching the sets which would essentially require to handle unbounded sets.


We now argue that conductance of Hit-and-Run with  $\beta$-approximate log-concave measures can be related to the conductance with log-concave measures.
\begin{theorem}[Conductance Lower Bound for Approximate Log-concave Measures]
	\label{thm:compare_cond}
	Let $\HH$ be a $\beta/2$-approximate log-concave measure and $h$ be any log-concave function with the property \eqref{eq:near_log_concave}. Then the conductance and $s$-conductance of the random walk induced by $\HH$ are lower bounded as
	$$\phi^\HH\geq e^{-3\beta} \phi^h ~~~~~\text{and}~~~~~\phi_s^\HH  \geq e^{-3\beta} \phi_{s/e^{\beta}}^h.
$$
\end{theorem}

We apply Theorem~\ref{thm:contraction_LV} to show contraction of $\sigma^{(m)}$ to $\pi_\HH$ in terms of the total variation distance.
\begin{theorem}[Mixing Time for Approximately-log-concave Measure]
	\label{thm:mixing_LC}
	Let $\pi_\HH$ is the stationary measure associated with the Hit-and-Run walk based on a $\beta/2$-approximately log-concave function $\HH$, and $M=\|\sigma^{(0)}/ \pi_\HH\| =  \int (d\sigma^{(0)}/d\pi_\HH) d\sigma^{(0)}$.
	There is a universal constant $C<\infty$ such that for any $\gamma \in (0,1/2)$, if$$
	m \geq C n^2 \frac{e^{6\beta}R^2}{r^2} \log^4 \frac{e^\beta MnR}{r \gamma^2} \log \frac{M}{\gamma}
	$$
	then $m$ steps of the Hit-and-Run random walk based on $\HH$ yield
	$$d_{\sf tv}(\pi_\HH, \sigma^{(m)}) \leq \gamma.$$
\end{theorem}
\begin{remark}
The value $M$ in Theorem~\ref{thm:mixing_LC} bounds the impact of the initial distribution $\sigma^{(0)}$ which can be potentially far from the stationary distribution. In the Simulated Annealing application of next section, we will show in Lemma \ref{lma: warm-start} that we can ``warm start'' the chain by carefully picking an initial distribution such that $M=\mathcal{O}(1)$.
\end{remark}

 Theorem \ref{thm:mixing_LC} shows $\gamma$-closeness between the distribution $\sigma^{(m)}$ and the corresponding stationary distribution. However, the stationary distribution is not exactly $\HH$ since the unidimensional sampling procedure described earlier truncates the distribution to improve mixing time.
The following theorem shows that these concerns are overtaken by the geometric mixing of the random walk.
Let $\hat \pi_{g,\ell}$ denote the distribution of the unidimensional sampling scheme (Algorithm~\ref{alg:rej_samp}) along the line $\ell$ and $\pi_{g,\ell}$ denote the distribution of the unidimensional sampling scheme proportional to $g$ along the line $\ell$.
\begin{theorem}\label{Thm:Allerrors}
Let $\hat{\sigma}^{(m)}$ denote the distribution of the Hit-and-Run with the unidimensional sampling scheme (Algorithm~\ref{alg:rej_samp}) after $m$ steps. For any $0<s<1/2$, the algorithm maintains that
$$d_{\sf tv}( \hat{\sigma}^{(m)}, \pi_{\HH}) \leq 2d_{\sf tv}( \hat{\sigma}^{(0)}, \sigma^{(0)}) + m \sup_{\ell \subset \K} d_{tv}(\hat \pi_{g,\ell},\pi_{g,\ell} ) + \left\{H_s + \frac{H_s}{s} \left(1 - \frac{(\phi_s^g)^2}{2}\right)^m\right\}	$$
where the supremum is taken over all lines $\ell$ in $\K$. In particular, for a target accuracy $\gamma \in (0,1/e)$, if $d_{\sf tv}( \hat{\sigma}^{(0)}, \sigma^{(0)}) \leq  \gamma /8$, $s$ such that $H_s \leq \gamma/4$, $m \geq \{2/(\phi_s^g)^2\} \log(\{H_s/s\}\{4/\gamma\})$, and the precision of the unidimensional sampling scheme to be $\tilde\epsilon= \gamma e^{-2\beta}/\{12m\}$, we have
$$d_{\sf tv}( \hat{\sigma}^{(m)}, \pi_{\HH}) \leq \gamma.$$
	\end{theorem}



\section{Optimization via Simulated Annealing}

We now turn to the main goal of the paper: to exhibit a method that produces an $\epsilon$-minimizer of the nearly convex function $F$ in expectation. Fix the pair $f,F$ with the property \eqref{eq:linfty_approx}, and define a series of functions
$$h_{i}(x) = \exp(-f/T_i), ~~~~~ \HH_i(x) = \exp(-F/T_i)$$
for a chain of temperatures $\{T_i, i= 1,\ldots,K\}$ to be specified later. It is immediate that $h_i$'s are log-concave. Lemma~\ref{lem:linfty_beta_concave}, in turn, implies that $\HH_i$'s are $\frac{2\epsilon}{nT_i}$-log-concave.

We now introduce the simulated annealing method that proceeds in epochs and employs the Hit-and-Run procedure with the unidimensional sampler introduced in the previous section. The overall simulated annealing procedure is identical to the algorithm of \cite{kalai2006simulated}, with differences in the analysis arising from $F$ being only approximately convex.

\vspace{1cm}
\begin{algorithm}[H]
\SetAlgoLined
\label{alg:annealing}
 \KwIn{A series of temperatures $\{T_i, 1\leq i \leq K\}$, $K$=number of epochs, $x \in {\rm int}\K$}
 \KwOut{a candidate point $x$ for which $F(x) \leq \min_{y \in \K}F(y) + \epsilon$ holds}
 initialization: well-rounded convex body $\K$ and $\{X_0^j, 1\leq j\leq N \}$ i.i.d. samples from uniform measure on $\K$, $N$-number of strands, set $\K_0 = \K$, and $\Sigma_0=I$  \;
 \While{$i$-th epoch, $1 \leq i \leq K$}{
   1. calculate the $i$-th rounding linear transformation $\mathcal{T}_i$ based on $\{X_{i-1}^j, 1\leq j\leq N \}$ and let $\Sigma_i = \mathcal{T}_i \circ \Sigma_{i-1}$ 
   \;
   2. draw $N$ i.i.d. samples $\{X_i^j, 1\leq j\leq N \}$  from measure $\pi_{\HH_i}$ using Hit-and-Run algorithm with linear transformation $\Sigma_i$ and with $N$ warm-starting points $\{X_{i-1}^j, 1\leq j\leq N \}$ \;
 }
 output ~~$x=\argmin{1\leq j\leq N, 1\leq i\leq K} F(X_i^j) $.
 \vspace{0.3cm}
 \caption{Simulated annealing}
\end{algorithm}
\vspace{0.5cm}

Before stating the optimization guarantee of the above simulated annealing procedure, we prove that the warm-start property of the distributions between successive epochs and the rounding guarantee given by $N$ samples.

\subsection{Warm start and mixing}

We need to prove that the measures between successive temperatures are not too far away in the $\ell_2$ sense, so that the samples from the previous epoch can be treated as a warm start for the next epoch. The following result is an extension of Lemma 6.1 in \citep{kalai2006simulated} to $\beta$-log-concave functions.
\begin{lemma}
\label{lma: warm-start}
	Let $\HH(x) = \exp(-F(x))$ be a $\beta$-log-concave function. Let $\mu_i$ be a distribution with density proportional to $\exp\{-F(x)/T_i\}$, supported on  $\K$. Let $T_i = T_{i-1}\left(1-\frac{1}{\sqrt{n}}\right)$. Then
	$$\|\mu_{i}/\mu_{i+1}\|\leq C_\gamma = 5 \exp(2\beta/T_i)$$
\end{lemma}

Next we account for the impact of using the final distribution from the previous epoch $\sigma^{(0)}$ as a ``warm-start."

\begin{theorem}\label{Thm:Warm}
	Fix a target accuracy $\gamma \in (0,1/e)$ and let $g$ be an $\beta/2$-approximately log-concave function in $\reals^n$. Suppose the simulated annealing algorithm (Algorithm~\ref{alg:annealing}) is run for $K=\sqrt{n} \log (1/\rho)$ epochs with temperature parameters $T_i=(1-1/\sqrt{n})^i, 0\leq i\leq K$. If the Hit-and-Run with the unidimensional sampling scheme (Algorithm~\ref{alg:rej_samp}) is run for $m=\mathcal{O}^*(n^3)$ number of steps prescribed in Theorem~\ref{thm:mixing_LC}, the algorithm maintains that
	\begin{align}
		\label{eq:keep_it_close2}
		d_{\sf tv}( \hat{\sigma}^{(m)}_{i} , \pi_{\HH_{i}}) \leq e\gamma
	\end{align}
	at every epoch $i$, where $\hat{\sigma}^{(m)}_{i}$ is the distribution of the $m$-th step of Hit-and-Run. Here, $m$ depends polylogarithmically on $\rho^{-1}$.
\end{theorem}

\subsection{Rounding for $\beta$-log-concave functions}
\label{sec:rounding}

The simulated annealing procedure runs $N=\mathcal{O}^*(n)$ strands of  random walk to round the log-concave distribution into near-isotropic position (say $1/2$-near-isotropic) at each temperature. The $N$ strands do not interact and thus the computation within each epoch can be parallelized, further reducing the time complexity of the algorithm. For $N$ i.i.d isotropic random vectors $X_i \in \mathbb{R}^n, 1\leq i\leq N$ sampled from a log-concave measure, the following concentration holds when $N$ is large enough:
\begin{align*}
	\sup_{\| v \|_{\ell_2}=1 } \left| v^T \left( \frac{1}{N} \sum_{i=1}^N X_i X_i^T \right) v - 1 \right| \leq \frac{1}{2}
\end{align*}
or, equivalently,
\begin{align}
	 \label{eq:spectral.concentration}
	 \frac{1}{2} \leq \sigma_{\min} \left(\frac{1}{N} \sum_{i=1}^N X_i X_i^T  \right) \leq \sigma_{\max} \left(\frac{1}{N} \sum_{i=1}^N X_i X_i^T  \right) \leq \frac{3}{2}.
\end{align}
Theorems of this type have been first achieved for uniform measures on the convex body $\K$ (measures with bounded Orlicz $\psi_2$ norm). \cite{bourgain1996random} proved this holds as long as $N\geq C n \log^3 n$. \cite{rudelson1999random} improved this bound to $N\geq C n \log^2 n$. For log-concave measures (with bounded Orlicz $\psi_1$ norm through Borell's lemma), \cite{guedon2007lp} proved a stronger version where $N \geq C n \log n$. See also \citep{adamczak2010quantitative} for further improvements. For bounded (almost surely) vectors, we can instead appeal to the the following literature. Theorem 5.41 in \citep{vershynin2010introduction} yields a spectral concentration bound for heavy tail random matricies with isotropic independent rows. (See also \cite{tropp2012user} Theorem 4.1 for matrix Bernstein's type inequalities.)

For our problem, we need to prove \eqref{eq:spectral.concentration} for independent near isotropic rows with $\beta$-log-concave measures. There are two ways to achieve this goal. The first is to invoke the \cite{guedon2007lp}'s result. Random vectors sampled from $\beta$-log-concave still belong to the Orlicz $\psi_1$ family, thus $N \geq \mathcal{O}(n \log n)$ is enough to achieve the goal with high probability. The second way is through the following lemma:
\begin{lemma}[\cite{vershynin2010introduction}, Theorem 5.41]
	Let $X$ be an $N \times n$ matrix whose rows $X_i$  are
	independent isotropic random vectors in $\mathbb{R}^n$. Let R be a number such that $\| X_i \|_{\ell_2} \leq R$ almost surely for all $i$. Then for every $t \geq 0$, one has
	$$\sqrt{N} - tR \leq \sigma_{\min} (X) \leq \sigma_{\max} (X) \leq \sqrt{N} + tR $$
	with probability at least $1 - 2n \exp(-c t^2)$, where $c > 0$ is an absolute constant.
\end{lemma}
Clearly if we take $N \geq  C R^2 \log n$, we have
$$
\frac{1}{2} \leq \sigma_{\min} \left(\frac{1}{N} \sum_{i=1}^N X_i X_i^T  \right) \leq \sigma_{\max} \left(\frac{1}{N} \sum_{i=1}^N X_i X_i^T  \right) \leq \frac{3}{2}
$$
with probability at least $1 - n^{-C}$, since $\K$ is uniformly bounded within $R = \mathcal{O}(\sqrt{n}) r$ and isotropic condition implies $r = \mathcal{O}(1)$ (which translates into $\| X_i \|_{\ell_2} \leq \mathcal{O}(\sqrt{n})$). Thus we conclude that $N = \mathcal{\Theta}(n\log n)$ is enough for bringing a $\beta$-log concave measure into isotropic position.

\subsection{Optimization guarantee}

We prove an extension of Lemma 4.1 in \citep{kalai2006simulated}:
\begin{theorem}
	\label{eq:opt_guarantee}
	Let $f$ be a convex function. Let $X$ be chosen according to a distribution with density proportional to $\exp\{-f(x)/T\}$. Then
	$$\E_f f(X) - \min_{x\in\K} f(x) \leq (n+1)T$$
	Furthermore, if $F$ is such that $|F-f|_\infty\leq \rho$, for $X$ chosen from a distribution with density proportional to $\exp\{-F(x)/T\}$, we have
	$$
	\E_F f(X) - \min_{x \in \K} f(x) \leq (n+1)T \cdot  \exp(2\rho/T)
	$$
\end{theorem}

The above Theorem implies that the final temperature $T_K$ in the simulated annealing procedure needs to be set as $T_K = \epsilon/n$. This, in turn, leads to $K=\sqrt{n}\log(n/\epsilon)$ epochs. The oracle complexity of optimizing $F$ is then, informally,

\begin{center}
	$\mathcal{O}^*(n^3)$ queries per sample ~~$\red{\times}$~~ $\mathcal{O}^*(n)$ parallel strands ~~$\red{\times}$~~ $\mathcal{O}^*(\sqrt{n})$ epochs ~~$ \red{=} ~~~~\mathcal{O}^*(n^{4.5})$
\end{center}

The following corollary summarizes the computational complexity result:
\begin{corollary}
	\label{cor:4.5}
	Suppose $F$ is approximately convex and $|F-f|\leq \epsilon/n$ as in \eqref{eq:linfty_approx}. The simulated annealing method with $K=\sqrt{n}\log(n/\epsilon)$ epochs produces a random point such that
	$$\En f(X) - \min_{x\in\K} f(x) \leq \epsilon,$$
	and thus
	$$\En F(X) - \min_{x\in\K} F(x) \leq 2\epsilon.$$
	Furthermore the number of oracle queries required by the method is $\mathcal{O}^*(n^{4.5})$. 
\end{corollary}

\section{Stochastic Convex Zeroth Order Optimization}
\label{sec:stoch_zeroth}

Let $f:\K\to\reals$ be the unknown convex $L$-Lipschitz funciton we aim to minimize. Within the model of convex optimization with stochastic zeroth-order oracle $\oracle$, the information returned upon a query $x\in\K$ is $f(x) + \epsilon_x$ where $\epsilon_x$ is the zero mean noise. We shall assume that the noise is sub-Gaussian with parameter $\sigma$. That is,
$$
\E \exp(\lambda \epsilon_x) \leq \exp(\sigma^2 \lambda^2/2).
$$
It is easy to see from Chernoff's bound that  for any $t \geq 0$
$$
\mathbb{P} (|\epsilon_x| \geq \sigma t) \leq 2\exp(-t^2/2).
$$
We can decrease the noise level by repeatedly querying at $x$. Fix $\tau>0$, to be determined later. The average $\bar{\epsilon}_x$ of $\tau$ observations is concentrated as
$$
\mathbb{P} \left( |\bar{\epsilon}_{x}| \geq \sigma t/\sqrt{\tau} \right) \leq 2\exp(-t^2/2).
$$
To use the randomized optimization method developed in this paper, we view  $f(x)+\bar{\epsilon}_x$ as the value of $F(x)$ returned upon a single query at $x$. Since the randomized method does not re-visit $x$ with probability $1$, the function $F$ is ``well-defined''.

Let us make the above discussion more precise by describing three oracles. Oracle $\oracle'$ draws noise $\epsilon_x$ for each $x\in\K$ prior to optimization. Upon querying $x\in\K$, the oracle deterministically returns $f(x)+\epsilon_x$, even if the same point is queried twice. Given that the optimization method does not query the same point (with probability one), this oracle is equivalent to an \emph{oblivious} version of oracle $\oracle$ of the original zeroth order stochastic optimization problem.

To define $\oracle_\alpha$, let $\net_\alpha$ be an $\alpha$-net in $\ell_\infty$ which can be taken as a box grid of $\K$. If $\K\subseteq R B_\infty$, the size of the net is at most $(R/\alpha)^n$. The oracle draws $\epsilon_x$ for each element  $x\in\net_\epsilon$, independently. Upon a query $x'\in\K$, the oracle deterministically returns $f(x)+\epsilon_x$ for $x\in\net_\alpha$ which is closest to $x'$. Note that $\oracle_\alpha$ is no more powerful than $\oracle'$, since the learner only obtains the information on the $\alpha$-net.

Oracle $\oracle_\alpha^\tau$ is a small modification of $\oracle_\alpha$. This modification models a repeated query at the same point, as described earlier. Parametrized by $\tau$ (the number of queries at the same point), oracle $\oracle_\alpha^\tau$ draws random variables $\epsilon_x$ for each $x\in \net_\alpha$, but sub-Gaussian parameter of $\epsilon_x$ is $\sigma/\sqrt{\tau}$. The optimization algorithm pays for $\tau$ oracle calls upon a single call to $\oracle_\alpha^\tau$.


We argued that $\oracle_\alpha^\tau$ is no more powerful than the original zeroth order oracle given that the algorithm does not revisit the point. In the rest of the section, we will work with $\oracle_\alpha^\tau$ as the oracle model. For any $x$, denote the projection to the $\net_\alpha$ to be $\mathcal{P}_{\net_\alpha} (x)$. Define $F:\K\mapsto \reals$ as
$$F(x) = f(\mathcal{P}_{\net_\alpha} (x))+\epsilon_{\mathcal{P}_{\net_\alpha} (x)}$$
where $\mathcal{P}_{\net_\alpha} (x)$ is the closest to $x$ point of $\net_\alpha$ in the $\ell_\infty$ sense. Clearly,
\begin{align}
	\label{eq:linfty_bound}
	|F-f|_\infty \leq \max_{x\in\net_\alpha} |\epsilon_x| + \alpha L
\end{align}
where $L$ is the ($\ell_\infty$) Lipschitz constant. Since $(\epsilon_x)_{x\in\net_\alpha}$ define a finite collection of sub-Gaussian random variables with sub-Gaussian parameter $\sigma$, we have that with probability at least $1-\delta$
$$\max_{x\in \net_\alpha} |\epsilon_x|\leq \sigma \sqrt{\frac{2n\log(R/\alpha)+2\log(1/\delta)}{\tau}}$$
From now on, we condition on this event, which we call $\mathcal{E}$. To guarantee  \eqref{eq:linfty_approx}, we set
$$\frac{\epsilon}{2n} = \sigma \sqrt{\frac{2n\log(R/\alpha)+2\log(1/\delta)}{\tau}} =  \alpha L$$
where $\tau$ is the parameter from oracle $\oracle_\alpha^{\tau}$. We use the first equality to solve for $\tau$ and the second to solve for $\alpha$:
$$\tau = \frac{\sigma^2 n^2(8n\log(R/\alpha)+8\log(1/\delta))}{\epsilon^2} = \frac{\sigma^2 n^2(8n\log(2LRn/\epsilon)+8\log(1/\delta))}{\epsilon^2} = \mathcal{O}^*(n^3/\epsilon^2)$$
and $\alpha=\epsilon/(2Ln)$.
Note here $L$ affects $\tau$ only logarithmically, and, in particular, we could have defined the Lipschitz constant with respect to $\ell_2$. We also observe that the oracle model depends on $\alpha$ and, hence, on the target accuracy $\epsilon$. However, because the dependence on $\alpha$ is only logarithmic, we can take $\alpha$ to be much smaller than $\epsilon$.

Together with the $\mathcal{O}^*(n^{4.5})$ oracle complexity proved in the previous section for optimizing $F$, the choice of $\tau=\mathcal{O}^*(n^3 \epsilon^{-2})$ evaluations per time step yields a total oracle complexity of
$$\mathcal{O}^*(n^{7.5}\epsilon^{-2})$$
for the problem of stochastic convex optimization with zeroth order information. We observe that a factor of $n^2$ in oracle complexity comes from the union bound over the exponential-sized discretization of the set. This (somewhat artificial) factor can be reduced or removed under additional assumptions on the noise, such as a draw from a Gaussian process with spatial dependence over $\K$. Alternatively, this $n^2$ factor could be removed completely if we could take a union bound over the polynomial number of points visited by the algorithm. Such an argument, however, appears to be tricky.

\section{Optimization of Non-convex Functions with Decreasing Fluctuations}
\label{sec:de-fluc}

Assume the non-convex function $F(x)$ has the property that the ``amount of non-convexity'' is decreasing as $x$ gets close to its global minimum $x^*$. If one has some control on the rate of this decrease, it is possible to break the optimization problem into stages, where at each stage one optimizes to the current level of non-convexity, redefines the optimization region to guarantee a smaller amount of ``non-convexity,'' and proceeds to the next stage. We are not aware of optimization methods for such a problem, and it is rather surprising that one may obtain provable guarantees through simulated annealing.

As one example, consider the problem of stochastic zeroth order optimization where the noise level decreases as one approaches the optimal point. Then, one would expect to obtain a range of oracle complexities between $\log(1/\epsilon)$ and $1/\epsilon^2$ in terms of the rate of the noise decrease.

Let us formalize the above discussion. Suppose there exists a $1$-Lipschitz $\alpha$-strongly convex function $f(x)$ with minimum achieved at $x^*\in\K$:
\begin{align*}
	f(x) - f(x^*) \geq \langle \nabla f(x^*), x - x^* \rangle + \frac{\alpha}{2} \| x - x^*\|^2 \geq \frac{\alpha}{2} \| x - x^*\|^2.
\end{align*}
Define a measure of ``non-convexity'' of $F$ with respect to $f$ in a ball of radius $r$ around $x^*$:
\begin{align*}
	\Delta(r) := \sup_{x \in \B_2^n(x^*,r)} |F(x) - f(x)|.
\end{align*}
We have in mind the situation where $\Delta(r)$ decreases as $r$ decreases to $0$.
At the $t$th stage of the optimization problem, suppose we start with an $\ell_2$ ball  $\B_2^n(x_{t-1}, 2r_t)$ of radius $2r_t$ with the property
$$
\B_2^n(x^*, 3r_{t}) \supset \B_2^n(x_{t-1}, 2r_{t}) \supset \B_2^n(x^*, r_{t}).
$$
Next, we run the simulated annealing procedure for the approximately log-concave function defined over this ball. After $\mathcal{O}^*(n^{4.5})$ queries, we are provided with a point $x_t$ such that in expectation (or high probability)
\begin{align*}
	f(x_t) - f(x^*) \leq  C n \cdot \Delta(3r_t)
\end{align*}
with some universal constant $C>0$. Thanks to strong convexity,
\begin{align*}
	\frac{\alpha}{2} \| x_t - x^*\|^2  \leq  f(x_t) - f(x^*)  \leq  C n \cdot \Delta(3r_t)
\end{align*}
which suggests the recursive definition of $r_{t+1}$:
\begin{align*}	
	\frac{\alpha}{2Cn} r_{t+1}^2 := \frac{\alpha}{2Cn}  \| x_t - x^*\|_{\ell_2}^2 \leq \Delta(3r_t).
\end{align*}
At stage $t+1$ we restrict the region to be $ \B_2^n(x_{t}, 2r_{t+1}) \supset \B_2^n(x^*, r_{t+1})$ and run the optimization algorithm again with the new parameter of approximate convexity. The recursion formula for the radius from $r_t$ to $r_{t+1}$ satisfies
$$
\frac{\alpha}{2Cn} r_{t+1}^2 \leq \Delta(3r_t).
$$
The recursion formula yields a fixed point --- a ``critical radius'' $r^*$ where no further improved can be achieved, with $\frac{\alpha}{2Cn} (r^*)^2 = \Delta(3r^*)$. Let us explore two examples:
\begin{align*}
	&\text{Polynomial}:~~~~ \Delta(r) = c r^{p}, ~0 < p < 2, \\
	&\text{Logarithmic}:~~~~ \Delta(r) = c \log(1+d r),
\end{align*}
where $c,d>0$ are constants. For the polynomial case, the critical radius is $r^* = \left(\frac{2 \cdot 3^pcC n}{\alpha}\right)^{\frac{1}{2-p}}$, and the required number of epochs is at most $\frac{\log \log (r_0/r^*)+ \log(1/\epsilon)}{\log (2/p)}$ if we want to get $r_t = (1+\epsilon) r$.
For the logarithmic case, the critical radius is the unique non-zero solution to
$$
\frac{2cC n}{\alpha} \log(1+3d r) =  r^2.
$$
We conclude that at an $\mathcal{O}^*(1)$ multiplicative overhead on the number of oracle calls, we can optimize to any level of precision above the fixed point $r^*$ of the non-convexity decay function.

\section{Further Applications}
\label{sec:apps}

In this section, we sketch several applications of the zeroth-order optimization method we introduced. Our treatment is cursory, meant only to give a sense of the range of possible domains.

\subsection{Private computation with distributed data}

Suppose $i=1,\ldots,n$ are entities---say, hospitals---that each possess private data in the form of $m$ covariate-response pairs  $\left\{(x_{i,j},y_{i,j})\right\}_{j=1}^{m}$. A natural approach to analyzing the aggregate data is to compute a minimizer $w^*$ of
\begin{align}
	\label{eq:reg_erm}
	f(w) = \frac{1}{mn}\sum_{i, j} \ell(x_{i,j},y_{i,j}; w) + R(w)
\end{align}
for some convex regularization function $R$ and a convex (in $w$) loss function $\ell$. For instance, $\ell(x_{i,j}, y_{i,j}; w)=(y_{i,j}-x_{i,j}\cdot w)^2$ and $R(w)=0$ would correspond to the problem of linear regression.

Given that the hospitals are not willing to release the data to a central authority that would perform the computation, how can the objective \eqref{eq:reg_erm} be minimized? We propose to use the simulated annealing method of this paper. To this end, we need to specify what happens when the value $f(w_t)$ at the current point $w_t$ is requested. Consider the following idea. The current $w_t$ is passed to a randomly chosen hospital $I_t\sim \text{unif}(1,\ldots,n)$. The hospital, in turn, privately chooses an index $J_t\sim \text{unif}(1,\ldots,m)$, computes the loss $\ell(x_{I_t,J_t}, y_{I_t,J_t}; w_t)$, adds zero-mean noise $\eta_t\sim N(0,1)$, and passes the resulting value
$$v_t = \ell(x_{I_t,J_t}, y_{I_t,J_t}; w_t) + \eta_t$$
back to the central authority. Since the computation is done privately by the hospital, the only value released to the outside world is the noisy residual. It is easy to check that $v_t$ is an unbiased estimate of $f(w_t)$:
$$\E[v_t] = f(w_t)$$
with respect to the random variables $(I_t,J_t)$ and $\eta_t$. Moreover, the noise level with respect to each source of randomness is of constant order. By repeatedly querying for the noisy value at $w_t$, the algorithm can reduce the noise variance, as in Section~\ref{sec:stoch_zeroth}, yet---importantly---the returned value is for a potentially different random choice of the hospital and the data point. This latter fact means that repeated querying does not allow the central authority to learn a specific data point. Interestingly, the additional layer of privacy given by the zero-mean noise $\eta_t$ presents no added difficulty to the minimization procedure, except for slightly changing a constant in the number of required queries.

\subsection{Two-stage stochastic programming}

\cite{dyer2013simple} discuss the following mathematical programming formulation:

\begin{align}
	\label{eq:two_stage}
	\max~~~ &px + \E\left[\max \left\{qy | Wy\leq Tx - \xi,~ y\in \reals^{n_1}\right\}\right] \\
	\text{subject to} ~~~&Ax\leq b, \notag
\end{align}
where $q\in\reals^{n_1}$, $W\in\reals^{d\times n_1}$, and $T\in\reals^{d\times n}$. The expectation is taken over the random variable $\xi$. This problem is concave in $x$, and can be solved in two stages. If, given $x$, an approximate value for the inner expected maximum can be computed, the problem falls squarely into the setting of zeroth order optimization with approximate function evaluations. While the method of \cite{dyer2013simple} is simpler, its dependence on the target accuracy $\epsilon$ is worse. Additionally, the method of this paper can deal directly with constraint sets with non-smooth boundaries; the method can also handle more general functions in \eqref{eq:two_stage} that are not smooth.

\subsection{Online learning via approximate dynamic programming}

Online learning is a generic name for a set of problems where the forecaster makes repeated predictions (or decisions). For concreteness, suppose that on each round $t=1,\ldots,T$, the forecaster observes some side information $s_t\in S$, makes a prediction $\widehat{y}_t\in\K$, and observes an outcome $y_t\in\Y$. The goal of the forecaster is to ensure small regret, defined as
$$\sum_{t=1}^T \ell(\widehat{y}_t,y_t) - \inf_{f\in\F}\sum_{t=1}^T \ell(f(s_t),y_t)$$
where $\F$ is a class of strategies, mapping $S$ to $\K$, and $\ell:\K\times\Y\to\reals$ is a cost function, which we assume to be convex in the first argument. The vast majority of online learning methods can be written as solutions to the following optimization problem (see \citealt{rakhlin2012relax}):
\begin{align*}
	\widehat{y}_t = \argmin{\widehat{y}\in\K} \max_{y_t\in\Y} \left\{ \ell(\widehat{y},y_t) + \Phi_t(s_1,y_1,\ldots,s_t,y_t)\right\}
\end{align*}
where $\Phi_t$ is a relaxation on the minimax optimal value. One of the tightest relaxations is the so-called sequential Rademacher complexity, which itself involves an expectation over a sequence of Rademacher random variables and a supremum over the class $\F$. While the gradient of $\Phi_t$ might not be available, it is often possible to approximately evaluate this function and solve the saddle point problem approximately.

\begin{appendix}

\section{Proofs of Section 3}

\begin{proof}[Proof of Lemma \ref{lem:linfty_beta_concave}]
The proof is straightforward:
	\begin{align*}
		g(\alpha x + (1-\alpha)y) &\geq e^{-\beta/2} h(\alpha x + (1-\alpha)y) \geq e^{-\beta/2} h(x)^\alpha h(y)^{1-\alpha} \\
					              &\geq e^{-\beta/2} (e^{-\beta/2}g(x))^\alpha (e^{-\beta/2}g(y))^{1-\alpha} \geq e^{-\beta} g(x)^\alpha g(y)^{1-\alpha}.
	\end{align*}
\end{proof}

\section{Proofs of Section 4.1}

\begin{proof}[Proof of Lemma \ref{Lemma:InitStepa}]

Consider a unidimensional $\beta$-log-concave function $g : \reals \rightarrow \reals$. In view of Lemma~\ref{lem:1dsandwich}, $g$ can be ``sandwiched'' by a log-concave function $h$ such that $e^{-\beta} h(x) \leq g(x) \leq h(x)$.

Given $\ell$, we want to find $p \in \ell$ such that $g(p) \geq e^{-3\beta} \max_{z\in \ell}g(z)$. We use the following 3-point method, inspired by \cite{AgaFosHsuKakRak13siam}, to provide such a point. Let us work with the convex function
\begin{align*}
\tilde{h} &= - \log h
\end{align*}
and a nearly-convex function $$\tilde{g} = - \log g.$$
The sandwiching guarantee can be written as
\begin{align*}
\tilde{h} (x) \leq \tilde{g}(x) \leq \tilde{h}(x) + \beta.
\end{align*}
We now claim that each iteration of the ``while'' loop of Algorithm~\ref{alg:initial} maintains the following property: either the length of the interval is reduced by at least $3/4$ while still containing the optimal point, or we have the output point $p$ that satisfies
\begin{align*}
	\tilde{g}(p) &\leq \min \tilde{g} + 3 \beta
\end{align*}
In the latter case, $g(p) \geq e^{-3\beta} \max_{z \in \ell} g(z)$ as desired.

There are essentially two cases. First, if $\tilde{g}(x_l) - \tilde{g}(x_r) > \beta$ (or similarly we can argue for $|\tilde{g}(x_l) - \tilde{g}(x_c)| > \beta$ and $|\tilde{g}(x_r) - \tilde{g}(x_c)| > \beta$), we have
$$
\tilde{h} (x_l) + \beta \geq \tilde{g} (x_l) > \tilde{g}(x_r) + \beta \geq \tilde{h} (x_r) +\beta
$$
and thus $\tilde{h}(x_l) > \tilde{h}(x_r)$. Because of convexity of $\tilde{h}$ we can safely remove $[\uu{x},x_l]$ with the remaining interval still containing the point we are looking for. Second case is when
\begin{align*}
|\tilde{g}(x_l) - \tilde{g}(x_r)| \leq \beta, ~~~~|\tilde{g}(x_l) - \tilde{g}(x_c)| \leq \beta, ~~~~|\tilde{g}(x_r) - \tilde{g}(x_c)| \leq \beta
\end{align*}
Here, we can show the function $g(x)$ is flat enough for $[\uu{x}, \bb{x}]$ and thus the best of $x_l, x_c, x_r$ are good enough. It is not hard to see that
\begin{align*}
|\tilde{h}(x_l) - \tilde{h}(x_r)| \leq 2\beta, ~~~~|\tilde{h}(x_l) - \tilde{h}(x_c)| \leq 2\beta, ~~~~|\tilde{h}(x_r) - \tilde{h}(x_c)| \leq 2\beta.
\end{align*}
Consider the point $x_l$. By convexity of $\tilde{h}$, there must be a supporting line $k_l(x)$ that is below the convex function $\tilde{h}$ and such that $k_l(x_l)=\tilde(x_l)$. Thus
\begin{align*}
\min_{[\uu{x},x_{l}]} \tilde{h}(x) \geq  \min_{[\uu{x},x_{l}]} k_l(x) \geq k_l(x_l) - 2\beta = \tilde{h}(x_l) - 2\beta
\end{align*}
using the fact that $|\uu{x}-x_{l}|=|x_l-x_c|$. Similarly we can prove
\begin{align*}
\min_{[x_{l},x_{r}]} \tilde{h}(x) \geq \tilde{h}(x_c) - 2\beta, ~~~~\min_{[x_{r},\bb{x}]} \tilde{h}(x) \geq \tilde{h}(x_r) - 2\beta .
\end{align*}
Thus
\begin{align*}
\min_{[\uu{x},\bb{x}]} \tilde{h}(x) \geq \min(\tilde{h}(x_l),\tilde{h}(x_c),\tilde{h}(x_r)) - 2\beta.
\end{align*}
By sandwiching
\begin{align*}
\min_{[\uu{x},\bb{x}]} \tilde{g}(x) & \geq \min(\tilde{g}(x_l),\tilde{g}(x_c),\tilde{g}(x_r)) - 3\beta
\end{align*}
and, hence,
\begin{align*}
\tilde{g}(p) &\leq \min_{[\uu{x},\bb{x}]} \tilde{g}(x) + 3\beta.
\end{align*}
It remains to show that the algorithm will terminate in an $\mathcal{O}^*(1)$ number of steps. Let $\tilde{L}$ be the Lipschitz constant of $\tilde{h}$. By the time the interval is shrunk to  $|\uu{x}-\bb{x}|\leq \beta/\tilde{L}$, the algorithm must have entered the second case above and terminated.
\end{proof}

\begin{proof}[Proof of Lemma \ref{Lemma:InitUni}]
Consider a unidimensional $\beta$-log-concave function $g : \reals \rightarrow \reals$. In view of Lemma~\ref{lem:1dsandwich}, $g$ can be ``sandwiched'' by a log-concave function $h$ such that $e^{-\beta} h(x) \leq g(x) \leq h(x)$.

We consider the interval $[x_l,x_r] = [p,\bb{x}]$ (the other case follows similarly). If $g(\bb{x}) \geq \frac{1}{2}e^{-\beta}\tilde{\epsilon} g(p)$, set $e_1 = \bb{x}$. Otherwise we have $g(\bb{x}) < \frac{1}{2}e^{-\beta} \tilde{\epsilon} g(p)$ and we proceed. The procedure always query the  midpoint $x_m$ of current interval $[x_l,x_r]$. If $g(x_m) > \tilde{\epsilon} g(p)$ set $x_l=x_m$, or if $g(x_m) <  \frac{1}{2}e^{-\beta} \tilde{\epsilon} g(p)$ set $x_r = x_m$, and continue the search. Either operation halves the interval. If the midpoint $x_m$ is such that $\frac{1}{2}e^{-\beta} \tilde{\epsilon} g(p) \leq g(x_m) \leq \tilde{\epsilon} g(p)$, stop the process and return $e_1 = x_m$.
At every iteration, the interval $[x_l,x_r]$ is such that $g(x_l)>\tilde \epsilon g(p)$ and $g(x_r) < \frac{1}{2}e^{-\beta} \tilde{\epsilon} g(p)$. We now claim that the algorithm must terminate in an $\mathcal{O}^*(1)$ number of steps.
Let $\tilde{h} = -\log h$, and let $\tilde{L}$ be the Lipschitz constant of $\tilde{h}$. As soon as the length of the current interval  $|x_l-x_r|<1/(2\tilde{L})$, we have $|\tilde{h}(x_l)-\tilde{h}(x_r)|<1/2$. Thus $h(x_l)/h(x_r) < e^{1/2}$ and $g(x_l)/g(x_r)< e^{1/2+\beta}$, implying that both $g(x_l)>\tilde \epsilon g(p)$ and $g(x_r) < \frac{1}{2}e^{-\beta} \tilde{\epsilon} g(p)$ cannot be true at the same time as $2e^{\beta}> e^{1/2+\beta}$. Hence, the algorithm terminates in a number of steps that is logarithmic in $\tilde{L}$.

\end{proof}

\begin{proof}[Proof of Lemma~\ref{lem:1dsampler}]
Let $h$ be the log-concave function associated with the $\beta$-log-concave function $g$ in the sense of Lemma~\ref{lem:1dsandwich}, so that $e^{-\beta}h(x)\leq g(x)\leq h(x)$ for all $x\in\ell$, and $L_f(t)$ denote the (upper) level set of a function $f$ at level $t$.

We note that since $L_g(t) \subset L_h(t)$ and $e^{-\beta}h(x)\leq g(x)\leq h(x)$, \eqref{eq:init_ei} implies that
	 	\begin{equation}\label{AuxEq}L_{h}(e^\beta\tilde{\epsilon} g(p)) \subseteq [e_{-1},e_{1}] \subseteq L_{h}(\mbox{$\frac{1}{2}$}e^{-\beta} \tilde{\epsilon} g(p)).\end{equation}
Moreover, either $e_{-1}=\uu{x}$ or $g(e_{-1})\leq \tilde\epsilon g(p)$ which implies $h(e_{-1})\leq e^\beta\tilde\epsilon g(p)<\frac{1}{2}e^{-\beta}g(p)$ if $\tilde\epsilon<\frac{1}{2}e^{-2\beta}$.

The stationary distribution for this sampling scheme is a truncated distribution according to the $\beta$-log-concave distribution $g$ restricted to $[e_{-1},e_{1}]$.  (Indeed, this correspond to the classic Accept-Reject method to simulate $g$ based on the uniform distribution with constant $M:=g(p)e^{3\beta}$, see \cite{robert2004monte} page 49.)
Therefore
$$
\begin{array}{rl}
d_{\sf tv}(\pi_{g,\ell},\hat \pi_{g,\ell}) & = d_{\sf tv}(\pi_{g,\ell}1\{\ell\setminus [e_{-1},e_1]\},\hat \pi_{g,\ell}1\{\ell\setminus[e_{-1},e_1]\}) + d_{\sf tv}(\pi_{g,\ell}1\{[e_{-1},e_1]\},\hat \pi_{g,\ell}1\{[e_{-1},e_1]\}) \\
&\leq  \mathbb{P}_{z\sim g} \left(z \notin [e_{-1},e_1] \right) +  \frac{\mathbb{P}_{z\sim g} \left(z \notin [e_{-1},e_1] \right)}{1-\mathbb{P}_{z\sim g} \left(z \notin [e_{-1},e_1] \right)}.
\end{array}
$$

Next we verify that the truncation error (in the total variation norm) of restricting $g$ to $[e_{-1},e_{1}]$ instead of $\ell$ is of the desired order. By Lemma~\ref{lem:tails} which quantifies the tail decay of unidimensional log-concave measures, we have
\begin{align*}
\mathbb{P}_{z\sim g} \left(z \notin [e_{-1},e_1] \right) & \leq e^\beta \cdot \mathbb{P}_{z\sim h} \left(z \notin [e_{-1},e_1]\right)  \leq e^\beta \cdot \mathbb{P}_{z\sim h} \left(z \notin L_{h}(e^\beta\tilde{\epsilon} g(p)) \right) \\
& \leq e^\beta \cdot \mathbb{P}_{z\sim h} \left(z \notin L_{h}(e^\beta\tilde{\epsilon} M_h) \right)  \leq e^\beta \cdot \mathbb{P}_{z\sim h}\left(h(z) \leq e^\beta  \tilde{\epsilon} M_h \right)  \leq e^{2\beta} \cdot \tilde{\epsilon}
\end{align*}
where we used (\ref{AuxEq}).

Thus, provided that $e^{2\beta}\tilde{\epsilon} \leq 1/2$, the total variation distance between the truncated measure $\hat\pi_{g,\ell}$ supported on $[e_{-1}, e_1]$ satisfies
$$
d_{\sf tv}(\pi_{g,\ell},\hat \pi_{g,\ell}) \leq  3e^{2\beta} \tilde{\epsilon}.
$$

In order to bound the number of evaluations we first bound the probability of the event  the event $\{g(x) > \frac{1}{2}e^{-2\beta} g(p)\}$.   Indeed, by (\ref{AuxEq}), if $X \sim U([e_{-1},e_{1}])$ and $Z \sim U\big(L_h(\mbox{$\frac{1}{2}$}\tilde\epsilon e^{-\beta} g(p))\big)$, we have
\begin{align*}
\mathbb{P}_{X} \left(g(X) \geq \frac{1}{2}e^{-2\beta}g(p)\right) & \geq \mathbb{P}_{X} \left(h(X) \geq \frac{1}{2}e^{-\beta} g(p) \right) \\
& \geq  \mathbb{P}_{Z} \left(h(Z) \geq  \frac{1}{2}e^{-\beta}h(p)\right) \\
& = \frac{{\sf vol}(L_h(\frac{1}{2}e^{-\beta} h(p)))}{{\sf vol}(L_h( \mbox{$\frac{1}{2}$}\tilde{\epsilon} e^{-\beta} h(p)))} \end{align*}
By Lemma~\ref{lem:volume} with $s = \mbox{$\frac{1}{2}$}\tilde{\epsilon} e^{-\beta} h(p)$ and $t=\mbox{$\frac{1}{2}$}e^{-\beta} h(p)$, it follows that
\begin{align*}
\frac{{\sf vol}(L_h(\mbox{$\frac{1}{2}$}e^{-\beta} h(p)))}{{\sf vol}(L_h( \mbox{$\frac{1}{2}$}\tilde{\epsilon} e^{-\beta} h(p)))} & \geq \frac{\log \frac{\max h}{\mbox{$\frac{1}{2}$}e^{-\beta}h(p)}}{\log \frac{\max h}{\mbox{$\frac{1}{2}$}\tilde{\epsilon} e^{-\beta} h(p)}}  =  \frac{\log \frac{\max h}{h(p)}+ \log 2 + \beta}{\log \frac{\max h}{ h(p)}+ \log 2 + \log 1/\tilde{\epsilon} + \beta}  \geq \frac{\log 2 + \beta}{\log 2 + \beta + \log 1/\tilde{\epsilon} }\geq \frac{\log 2}{\log (2/\tilde{\epsilon})}.\end{align*}
Then, since $r\sim U([0,1])$ we have
\begin{align*}
& \mathbb{P}\left(r \leq \frac{e^{-3\beta} g(x)}{g(p)} \mid g(x) \geq \frac{1}{2}e^{-2\beta} g(p) \right) \geq \frac{1}{2}e^{-5\beta}
\end{align*}
and thus
\begin{align*}
\mathbb{P}\left(r \leq \frac{e^{-3\beta} g(x)}{g(p)} \right) \geq \frac{e^{-5\beta}\log 2}{2\log (2/\tilde{\epsilon})}.
\end{align*}

Since we have a lower bound on the acceptance probability on each sampling step, the number of iterations we need to sample is of the order $\frac{e^{-5\beta}\log 2}{2\log (2/\tilde{\epsilon})}$. This quantity is $\mathcal{O}(\log(1/\tilde\epsilon))$ if $\beta$ is $\mathcal{O}(1)$.
\end{proof}

\section{Proofs of Section 4.2}

\begin{proof}[Proof of Theorem \ref{thm:compare_cond}]
Define the shorthand $\rho=e^{\beta/2}$. By sandwiching,
$$
\rho^{-1} h(x) \leq \HH (x) \leq  \rho h(x)
$$
Then
\begin{align*}
\rho^{-2} \pi_h(x) \leq \pi_\HH(x) \leq \rho^2 \pi_h(x)  \\
\rho^{-2} P_u^h(A) \leq P_u^\HH(A) \leq \rho^2 P_u^h(A)
\end{align*}
for any $x, u \in \K$ and $A \subset \K$.
Thus we have
$$
\phi^{\HH}(S) \geq \rho^{-6} \phi^{\HH}(S).
$$
The $s$-conductance bound can be derived as follows.
\begin{align*}
\phi_s^\HH & = \inf_{A \subset \K, s \leq \pi_\HH(A) \leq 1/2} \frac{\int_{x \in A} P_x^\HH(K \backslash A) d \pi_\HH}{\pi_\HH(A) - s} \\
  &\geq \rho^{-6}  \inf_{A \subset \K, s \leq \pi_\HH(A) \leq 1/2} \frac{\int_{x \in A} P_x^h(K \backslash A) d \pi_h}{\pi_h(A) - s/\rho^2} \\
 &\geq \rho^{-6}  \inf_{A \subset \K, s/\rho^2 \leq \pi_h(A) \leq 1/2} \frac{\int_{x \in A} P_x^h(K \backslash A) d \pi_h}{\pi_h(A) - s/\rho^2} = \rho^{-6} \phi_{s/\rho^2}^h
\end{align*}

\end{proof}

\begin{proof}[Proof of Theorem \ref{thm:mixing_LC}]
	The $H_s$ defined in Theorem~\ref{thm:contraction_LV} can be upper bounded by
\begin{align*}
H_s & = \sup_{A:\pi_\HH(A) \leq s} \int_{A} \left| \frac{d \sigma^{(0)}}{d \pi_\HH } - 1 \right| d \pi_\HH \\
    & \leq \sup_{A:\pi_\HH(A) \leq s} \left\{ \int_{A} \left( \frac{d \sigma^{(0)}}{d \pi_\HH } - 1 \right)^2 d \pi_\HH  \cdot \int_{A} d \pi_\HH \right\}^{1/2} \\
	& \leq s^{1/2} \sup_{A:\pi_\HH(A) \leq s} \left\{ \int_{A} \left( \frac{d \sigma^{(0)}}{d \pi_\HH } - 1 \right)^2 d \pi_\HH \right\}^{1/2} \leq s^{1/2} \| \sigma^{(0)}/\pi_\HH \|^{1/2} = s^{1/2} M^{1/2}.
\end{align*}
Let us now use upper bound of Theorem~\ref{thm:contraction_LV} with $s = \left(\frac{\gamma}{2M^{1/2}}\right)^2$ and $D = \frac{2R}{r}$, as well as Theorem~\ref{thm:compare_cond}. We obtain
\begin{align*}
d_{tv}(\pi_\HH,\sigma^{(m)}) & \leq \frac{\gamma}{2} + \frac{2M}{\gamma} \left(1 - \frac{(\phi_{s}^\HH)^2}{2} \right)^m \\
& \leq \frac{\gamma}{2}  + \frac{2M}{\gamma} \left(1 -  \frac{(\rho^{-6}\phi_{s/\rho^2}^h)^2}{2} \right)^m \\
& \leq \frac{\gamma}{2} + \frac{2M}{\gamma} \exp\left( - \frac{m (\rho^{-6}\phi_{s/\rho^2}^h)^2  }{2} \right)
\end{align*}
where $\rho=e^{\beta/2}$. In view of Theorem~\ref{thm:cond_lower_bound},
$$\phi_{s/\rho^2}^h \geq \frac{cr}{nR\log^2\left(\frac{\rho^2 nR}{r\left(\gamma/2M^{1/2}\right)^2}\right)}$$
we arrive at
\begin{align*}
d_{tv}(\pi_\HH,\sigma^{(m)}) & \leq \frac{\gamma}{2} + \frac{2M}{\gamma} \exp\left( - \frac{m}{2} \left(\frac{\rho^{-6}c r}{n R \log^2 \frac{\rho^2 n R M}{r\gamma^2}} \right)^2 \right)
\end{align*}
Hence, if
$$
m \geq C n^2 \frac{\rho^{12}R^2}{r^2} \log^4 \frac{\rho^2 M nR}{r \gamma^2} \log \frac{M}{\gamma}
$$
then
$$d_{\sf tv}(\pi_\HH, \sigma^{(m)}) \leq \gamma.$$
\end{proof}

\begin{proof}[Proof of Theorem \ref{Thm:Allerrors}]
Step 1. (Main Step) By the triangle inequality we have
\begin{equation}\label{TriIneq}
d_{\sf tv}(\hat \sigma^{(m)}, \pi_\HH) \leq d_{\sf tv}(\hat\sigma^{(m)}, \sigma^{(m)}) + d_{\sf tv}(\sigma^{(m)}, \pi_\HH)
\end{equation}
The last term in (\ref{TriIneq}) converges to zero at a geometric rate in $m$ by Theorem~\ref{thm:mixing_LC}. Specifically we have
$$d_{\sf tv}(\sigma^{(m)}, \pi_\HH) \leq H_s + \frac{H_s}{s} \left(1 - \frac{(\phi_s^g)^2}{2}\right)^m.$$

To bound $d_{\sf tv}(\hat\sigma^{(m)}, \sigma^{(m)})$, the total variation distance after $m$ steps between the two random walks from their corresponding starting distributions $\hat{\sigma}^{(0)}$ and $\sigma^{(0)}$, write for any measurable set $A$
$$
\hat{\sigma}^{(m)}(A) = \int (P_x^{\hat{\HH}})^{(m)}(A) \hat{\sigma}^{(0)} d x \ \ \mbox{and} \ \ \sigma^{(m)}(A) = \int (P_x^{\HH})^{(m)}(A) \sigma^{(0)} d x
$$
so that
\begin{align*}
	d_{\sf tv}\left( \sigma^{(m)}, \hat{\sigma}^{(m)}\right) \leq \sup_{u\in \K} ~ d_{\sf tv}\left((P_u^{\hat{\HH}})^{(m)}, (P_u^{\HH})^{(m)} \right) + 2d_{\sf tv}\left(\hat{\sigma}^{(0)},\sigma^{(0)}  \right).
\end{align*}

The result follows from Step 2 that shows $\sup_{u\in \K} ~ d_{\sf tv}\left((P_u^{\hat{\HH}})^{(m)}, (P_u^{\HH})^{(m)} \right) \leq m \sup_{\ell \subset \K} d_{tv}(\hat \pi_{g,\ell},\pi_{g,\ell} ).$

Step 2. (Error Propagation Bound in $m$ Steps) The unidimensional sampling scheme produces a sample from a truncated distribution (see Lemma~\ref{lem:1dsampler}). That is, at each step of the Hit-and-Run algorithm, we are sampling from a truncated measure according to a truncated function $\hat{\HH}$ along each line $\ell$ of $\HH$ (approximately-log-concave function in $\mathbb{R}^n$). Let us denote the transition probability kernel starting from $u$ for this truncated function to be $P_u^{\hat{\HH}}$
and the kernel for the original function is $P_u^{\HH}$. Let us bound the total variation distance between these two kernels through the spherical (elliptical) coordinate system, and with $p(\cdot)$ being the density corresponding to the measure $P(\cdot)$).

Suppose that $\sup_{\ell \subset \K} d_{tv}(\hat \pi_{g,\ell},\pi_{g,\ell} ) \leq \tilde\epsilon$.
Since $p_u^{\hat{\HH}}(\theta) = p_u^{\HH}(\theta)=p(\theta)$, it holds that
\begin{align*}
	 2 d_{\sf tv}(P_u^{\hat{\HH}}, P_u^{\HH}) & = \int |p_u^{\hat{\HH}}(r|\theta)p_u^{\hat{\HH}}(\theta) - p_u^{\HH}(r|\theta) p_u^{\HH}(\theta)|  dr d\theta \\
	& \leq \int \left\{ \int |p_u^{\hat{\HH}}(r|\theta) - p_u^{\HH}(r|\theta)|  dr \right\} p(\theta) d\theta   \\
	& \leq 2\tilde{\epsilon} \int p(\theta) d\theta = 2\tilde{\epsilon}
\end{align*}
where on each line (over all $\theta$ according to the measure given by the linear transformation composed with uniform direction) the truncated distribution is an $\tilde{\epsilon}$ approximation to $\pi_\HH$. We now claim that the $m$-fold iterate of the Hit-and-Run kernel satisfies $d_{\sf tv}\left((P_u^{\hat{\HH}})^{(m)}, (P_u^{\HH})^{(m)}\right) \leq m \tilde{\epsilon}$. Let us prove this by induction. Suppose it holds for $m-1$ steps. Then
\begin{align*}
	2 d_{\sf tv}\left((P_u^{\hat{\HH}})^{(m)}, (P_u^{\HH})^{(m)} \right) &= \int | \int (p_u^{\hat{\HH}})^{(m-1)}(y) p_y^{\hat{\HH}}(x) - (p_u^\HH)^{(m-1)}(y) p_y^\HH(x)  dy | dx  \\
	&\leq  \int | \int (p_u^{\hat{\HH}})^{(m-1)}(y) p_y^{\hat{\HH}}(x) - (p_u^\HH)^{(m-1)}(y) p_y^{\hat{\HH}}(x)  dy | dx \\
	& \quad + \int | \int (p_u^\HH)^{(m-1)}(y) p_y^{\hat{\HH}}(x) - (p_u^\HH)^{(m-1)}(y) p_y^\HH(x)  dy | dx \\
	& \leq \int |(p_u^{\hat{\HH}})^{(m-1)}(y)  - (p_u^\HH)^{(m-1)}(y) | \left( \int p_y^{\hat{\HH}}(x)  dx \right)  dy \\
	& \quad + \int (p_u^\HH)^{(m-1)}(y) \left( \int| p_y^{\hat{\HH}}(x) -  p_y^\HH(x)| dx\right) dy \\
	& \leq 2 d_{\sf tv}\left((P_u^{\hat{\HH}})^{(m-1)}, (P_u^{\HH})^{(m-1)} \right) + 2 \max_y ~ d_{\sf tv}(P_y^{\hat{\HH}}, P_y^{\HH}) \\
    & \leq 2(m-1)\tilde{\epsilon} + 2\tilde{\epsilon} =  2m\tilde{\epsilon}
\end{align*}
\end{proof}

\begin{proof}[Proof of Theorem \ref{thm:cond_lower_bound}]
The proof follows closely the arguments in the proof of Theorem 3.7 in \cite{lovasz2006fast} for bounded sets with modifications to avoid the truncation device discussed in Section 3.3 of \cite{lovasz2006fast}. Define the step-size $F(x)$ by $P(\|x-y\|\leq F(x))=1/8$ where $y$ is a random step from $x$. Next define $\lambda(x,t)= {\rm vol}((x+tB)\cap L(\frac{3}{4}f(x)))/{\rm vol}(tB)$ and $s(x)=\sup\{t>0:\lambda(x,t)\geq 63/64\}$. Finally, $\alpha(x) = \inf\{ t \geq 3 :P(f(y)\geq t f(x)) \leq 1/16\}$ where $y$ is a hit-and-run step from $x$.

Let $\K = S_1 \cup S_2$ be a partition into measurable sets, where $S_1  = S$ and $p = \pi_h(S_1) \leq \pi_h(S_2)$. For for $D = R \log (C' n^2/p )$ we will prove that
\begin{align}
\label{cond.eqtn}
\int_{S_1} P_x(S_2) dx \geq \frac{1}{C nD \log \frac{nD}{p}} \pi_h(S_1).
\end{align}
Consider the points that are deep inside these sets with respect to 1-step distribution
$$
S_1' = \{ x \in S_1 : P_x(S_2) < 1/1000 \}
\ \
\mbox{and}
\ \
S_2' = \{ x \in S_2: P_x(S_1) < 1/1000 \},
$$
and the complement $S_3' = \K \backslash S_1' \cup S_2'$.

Suppose $\pi_h(S_1') < \pi_h(S_1)/2$. Then
$$
\int_{S_1} P_x(S_2) dx \geq \frac{1}{1000} \pi_h(S_1 \backslash S_1') \geq \frac{1}{2000} \pi_h(S_1)
$$
which proves \eqref{cond.eqtn}. Thus we can assume
\begin{align}
\pi_h(S_1') \geq \pi_h(S_1)/2 \ \ \mbox{and} \ \
\pi_h(S_2') \geq \pi_h(S_2)/2.
\end{align}

Define the exceptional subset $W$ as set of points $u$ for which $\alpha(u)$ is very large
$$
W = W_1 \cup W_2, \ \ \mbox{where} \ W_1 := \{ u \in S:\alpha(u) \geq 2^{27} nD/p\} \ \ \mbox{and} \ \ W_2 := \{u \in \K : \|x-z_h\| \geq D \}.
$$
By Lemma 6.10 in \cite{lovasz2007geometry}, $\pi_h(W_1) \leq p/\{2^{23} nD\}$ and, by Lemma 5.17 in \cite{lovasz2007geometry}, $\pi_h(W_2) \leq \frac{p}{Cn^2}$. Now for any $u \in S_1' \backslash W$ and $v \in S_2' \backslash W$
$$
d_{tv}(P_u,P_v) \geq P_u(S_1) - P_v(S_1) = 1 - P_u(S_2) -P_v(S_1) > 1- \frac{1}{500}
$$
from the definition of $S_1'$ and $S_2'$. Thus by Lemma 6.8 of \cite{lovasz2006hit}, we have
$$
d_h(u,v) \geq \frac{1}{128 \log (3 + \alpha(u))} \geq \frac{1}{2^{12} \log \frac{nD}{p}} \ \ \mbox{or} \ \
|u - v| \geq \frac{1}{4\sqrt{n}} \max \{ F(u), F(v) \}.
$$
By Lemma 3.2 of \cite{lovasz2006hit}, the latter implies that
$$
|u - v| \geq \frac{1}{2^8 \sqrt{n}} \max \{ s(u), s(v) \}
$$
In either case, by Lemma 3.5 in \cite{lovasz2006fast}, for any point $x \in [u, v]$, we have
\begin{align*}
s(x) & \leq 2^{14} \log \frac{nD}{p} |u-v|\sqrt{n} \\
& \leq 2^{14} \log \frac{nD}{p} d_{\K\setminus W_2}(u,v) D \sqrt{n}
\end{align*}
where the second inequality follows from $u,v \in \K\setminus W_2$.

Recall the original partition $S_1', S_2'$ and $S_3'=\K\setminus \{S_1'\cup S_2'\}$ of $\K$. We will apply Theorem 2.1 of \cite{lovasz2006hit} with a different partition. Consider the partition of $ \K \setminus W_2$ defined as $\bar S_1 = S_1' \backslash W$, $\bar S_2 = S_2' \backslash W$ and $\bar S_3=\K\setminus \{W_2 \cup S_1'\cup S_2'\}$. These definitions imply that $\bar S_3 \subset S_3' \cup W$ so that \begin{equation}\label{S3rel}\pi_h(\bar S_3) \leq \pi_h(S_3')+\pi_h(W).\end{equation} Define for $x \in \K \setminus W_2$
$$
\bar h(x)  = \frac{s(x)}{2^{16} D \sqrt{n} \log \frac{nD}{p}}
$$
It follows that for any $u \in S_1'\backslash W$ and $v \in S_2'\backslash W$ and $x \in [u,v]$,  we have
$\bar h(x) \leq d_{\K\setminus W_2}(u,v)/3$, and since $\K\setminus W_2$ is a convex body, we have by Theorem 2.1 of \cite{lovasz2006hit} that
\begin{equation}\label{BigStep} \int_{\bar S_3} h(x) dx \geq \frac{\int_{\K\setminus W_2} \bar h(x) h(x) dx }{\int_{\K\setminus W_2} h(x) dx} \min\left\{\int_{\bar S_1} h(x) dx,\int_{\bar S_2} h(x) dx \right\} \end{equation}

Although $\mathbb{E}_h(s(x))$ is large, we need a lower bound on $\int_{\K\setminus W_2} s(x) h(x) dx/ \int_\K h(x)dx$. Since $s(x)$ can be large if $\K$ is unbounded we modify the standard bound next. Because the level set of measure $1/8$
contains a ball of radius $r$, we have
$$\begin{array}{rl}
\int_{\K\setminus W_2} s(x) h(x) dx/\int_\K h(x)dx &=  \int_{\K\setminus W_2} \int_0^{s(x)} dt h(x) dx /\int_\K h(x)dx\\
& = \int_0^\infty \int_{\{ x \in \K\setminus W_2: \lambda(x,t) \geq 63/64 \}} h(x) dx dt /\int_\K h(x)dx\\
& \geq  \frac{1}{16}\int_0^\infty \int_{\{x \in \K\setminus W_2: \lambda(x,t)\geq 3/4\}}h(x)dxdt/\int_\K h(x)dx\\
& \geq \frac{1}{16}\int_0^\infty \left(\int_{\{x \in \K: \lambda(x,t)\geq 3/4\}}h(x)dx - \int_{W_2}h(x)dx\right)_+dt/\int_\K h(x)dx \\
&  \geq \frac{1}{16}\int_0^\infty \left(\frac{1}{2}-\frac{12t\sqrt{n}}{r} - \frac{p}{Cn^2}\right)_+dt \\
& \geq \frac{r}{2^{12}\sqrt{n}}
 \end{array}$$
where the first and third inequality follows from page 998 in \cite{lovasz2006hit}, the second by definition of $W_2$ where we take $C$ and $n$ large enough. Therefore, dividing both sides of (\ref{BigStep}) by $\int_\K h(x) dx$, (\ref{S3rel}), $\pi_h(S_i') \geq \pi_h(S_i)/2$, $i=1,2$,  and $p=\pi_h(S_1)\leq \pi_h(S_2)$,  we have
\begin{align*}
\pi_h(S_3') + \pi_h(W) \geq \pi_h(\bar S_3) &\geq \frac{\int_{\K\setminus W_2} \bar h(x) h(x) dx }{\int_{\K\setminus W_2} h(x) dx} \min\{\pi_h(\bar S_1), \pi_h(\bar S_2)\}\\
 & = \frac{\int_{\K\setminus W_2} s(x) h(x) dx/\int_\K h(x)dx}{2^{16} D \sqrt{n} \log \frac{nD}{p}} \min\{\pi_h(S_1' \backslash W), \pi_h(S_2' \backslash W)\}\\
& \geq \frac{1}{2^{28} n(D/r) \log \frac{nD}{p}}\left\{\min\{\pi_h(S_1'), \pi_h(S_2')\}  - \pi_h(W) \right\}\\
& \geq \frac{1}{2^{28} n(D/r) \log \frac{nD}{p}}\left\{\frac{1}{2}\min\{\pi_h(S_1), \pi_h(S_2)\}  - \pi_h(W) \right\}\\
& \geq \frac{1}{2^{28} n(D/r) \log \frac{nD}{p}}\left\{\frac{\pi_h(S_1)}{2}  - p/4 \right\} \geq \frac{\pi_h(S_1)}{2^{32} n(D/r) \log \frac{nD}{p}}\\
\end{align*}
where we used that $\pi_h(W) \leq p/4$. Therefore,
$$ \int_{S_1} P_x(S_2)dx \geq \frac{\pi_h(S_3')}{2000}\geq \frac{\pi_h(S_1)}{Cn(D/r)\log(nD/p)} $$
\end{proof}

\section{Proofs of Section 5}

\begin{proof}[Proof of Lemma \ref{lma: warm-start}] Define
$$
Y(a) = \int_{\K} e^{-F(x)a} dx
$$
With this notation, we have that
$$
\left\| \frac{\mu_i}{\mu_{i+1}} \right\| = \frac{Y(2/T_i - 1/T_{i+1}) Y(1/T_{i+1})}{Y(1/T_i)^2}
$$
Define
$$
G(x,t) = \HH\left(\frac{x}{t} \right)^t.
$$
Then it holds that
\begin{align*}
G(\lambda(x,t)+(1-\lambda)(x',t')) & = \HH\left(\frac{\lambda x+ (1-\lambda) x'}{\lambda t + (1 - \lambda)t'} \right)^{\lambda t + (1 - \lambda)t'} \\
& =  \HH\left(\frac{\lambda t}{\lambda t + (1 - \lambda)t'} \frac{x}{t} + \frac{(1-\lambda) t'}{\lambda t + (1 - \lambda)t'} \frac{x'}{t'} \right)^{\lambda t + (1 - \lambda)t'} \\
& \geq \exp(-\beta(\lambda t + (1-\lambda) t')) \HH\left( \frac{x}{t} \right)^{\lambda t} \HH\left( \frac{x'}{t'} \right)^{(1-\lambda) t'} \\
& = \exp(-\beta(\lambda t + (1-\lambda) t')) G(x,t)^\lambda G(x',t')^{(1-\lambda)}
\end{align*}
Because
$$
\int_{\K} G(x,t) dx = \int_{\K} \HH\left( \frac{x}{t} \right)^t dx  = t^n \int_{\K} \HH(x)^t dx = t^n Y(t)
$$
through Pr\'{e}kopa-Leindler inequality, we have
$$
\left(\frac{a+b}{2}\right)^{2n} Y\left( \frac{a+b}{2} \right)^2 \geq \exp(-\beta(a+b))a^n Y(a) b^n Y(b)
$$
Take $a = 2/T_i - 1/T_{i+1}$ and $b = 1/T_{i+1}$. Then we have
\begin{align*}
\left\| \frac{\mu_i}{\mu_{i+1}} \right\| &= \frac{Y(2/T_i - 1/T_{i+1}) Y(1/T_{i+1})}{Y(1/T_i)^2} \\
& \leq \left( \frac{1/T_i^2}{(2/T_i - 1/T_{i+1})(1/T_{i+1})}  \right)^n \exp(2\beta/T_i) \\
&  = \left( 1+\frac{1}{n - 2\sqrt{n}} \right)^n \exp(2\beta/T_i) \leq e^{n/(n-2\sqrt{n})} \exp(2\beta/T_i) \leq 5 \exp(2\beta/T_i)
\end{align*}
\end{proof}

\begin{proof}[Proof of Theorem \ref{Thm:Warm}]
Let us prove \eqref{eq:keep_it_close2} by induction. Suppose at the end of epoch $i$, we have
$$
d_{\sf tv}( \hat{\sigma}^{(m)}_{i} , \pi_{\HH_i}) \leq \left(1+\frac{1}{n\log 1/\rho}\right)^i \gamma
$$
We identify $\hat{\sigma}^{(m)}_{i} = \hat{\sigma}^{(0)}_{i+1}$,  $\sigma^{(0)}_{i+1} = \pi_{\HH_{i}}$. Hence, by Step 2 in the proof of Theorem \ref{Thm:Allerrors}, we have
\begin{align*}
d_{\sf tv}( \hat{\sigma}^{(m)}_{i+1} , \pi_{\HH_{i+1}})
&\leq d_{\sf tv}( \sigma^{(m)}_{i+1} , \pi_{\HH_{i+1}}) + d_{\sf tv}( \hat{\sigma}^{(m)}_{i+1} ,\sigma^{(m)}_{i+1}) \\
&\leq d_{\sf tv}( \sigma^{(m)}_{i+1} , \pi_{\HH_{i+1}}) + m\tilde{\epsilon} + d_{\sf tv}(\hat{\sigma}^{(0)}_{i+1},\sigma^{(0)}_{i+1} ) \\
& =d_{\sf tv}( \sigma^{(m)}_{i+1} , \pi_{\HH_{i+1}}) + m\tilde{\epsilon} + 2d_{\sf tv}( \hat{\sigma}^{(m)}_{i} , \pi_{\HH_{i}})  \\
&\leq \left(1+\frac{1}{n\log 1/\rho}\right)^i  \gamma \cdot \frac{1}{4n \log 1/\rho} +\left(1+\frac{1}{n\log 1/\rho}\right)^i  \gamma \cdot \frac{1}{4n \log 1/\rho} + \left(1+\frac{1}{n\log 1/\rho}\right)^i  \gamma  \\
&\leq \left(1+ \frac{1}{n\log 1/\rho}\right)^{i+1} \gamma
\end{align*}
by choosing $m = \mathcal{O}^*(n^3)$ and $$\tilde{\epsilon} =\frac{1}{m} \cdot \left(1+\frac{1}{n\log 1/\rho}\right)^i \gamma \cdot \frac{1}{4n\log 1/\rho}.$$

Thus the final epoch, $i=\sqrt{n}\log 1/\rho$, the error is at most $\left(1+\frac{1}{n\log 1/\rho}\right)^{\sqrt{n}\log 1/\rho} \gamma \leq e \gamma$.
\end{proof}

\begin{proof}[Proof of Theorem \ref{eq:opt_guarantee}]
	In \cite{kalai2006simulated}, the authors proved the theorem for the case $f(x) = c\cdot x$ and claimed it can be extended to arbitrary convex functions. Here we give a proof for the sake of completeness. \cite{kalai2006simulated} proved above inequality with arbitrary convex set $K$. Let's see how to relate an arbitrary convex function $f(x)$ to a linear function by increasing the dimension by 1. Consider a convex set $\K \in \mathbb{R}^n$ and a continuous convex function $f: \mathbb{R}^n \rightarrow \mathbb{R}$. Consider the epigraph
	$$
	\tilde{\K}:=\{(x,y): x\in \K, y \geq f(x) \}.
	$$
Define the linear function $\tilde{f}(x,y) = c \cdot (x,y)$, where $(x,y)\in \tilde{\K}$ and $c = ({\bf 0},1) \in \mathbb{R}^{n+1}$. We have
$$
\mathbb{E} \tilde{f}(X,Y) \geq \mathbb{E} f(X) \quad \text{the first one increase mass on large values}
$$
$$
\min_{(x,y)\in \tilde{\K}} \tilde{f}(x,y) = \min_{x \in \K} f(x)
$$
Thus
$$
\E f(X) - \min_{x\in\K} f(x) \leq \E \tilde{f}(X,Y) - \min_{(x,y)\in \tilde{\K}} \tilde{f}(x,y) \leq (n+1)T
$$
proof completed.

For the second claim, it is not hard to see that
$$
\E_F f(X) - \min_{x \in \K}f(x) = \E_F \left[f(X) - \min_{x \in \K}f(x)\right].
$$
Since adding a constant to the function does not have an effect on the density, we can assume without loss of generality that $\min_{x \in K}f(x) = 0$. Thus we have
\begin{align*}
\E_F f(X) &= \frac{\int_{\K} f(x) \exp\{-F(x)/T\} dx}{\int_{\K} \exp\{-F(x)/T\} dx} \leq \frac{\int_{\K} f(x) \exp\{-f(x)/T\} dx \cdot \exp(\rho/T) }{\int_{\K} \exp\{-f(x)/T\} dx \cdot \exp(-\rho/T)} \\& \leq \exp(2\rho/T) \cdot \E f(X)  \leq (n+1)T \cdot \exp(2\rho/T)
\end{align*}
and
$$
\E_F f(X) - \min_{x\in\K} f(x) \leq (n+1)T \cdot  \exp(2\rho/T).
$$
\end{proof}	

\begin{proof}[Proof of Corollary \ref{cor:4.5}]
	We choose $\rho =\epsilon/n$ and $T_K=\rho$. Given the final temperature, $K=\sqrt{n}\log(n/\epsilon)$. The optimization guarantee follows from Theorem~\ref{eq:opt_guarantee}. The number of queries is $\mathcal{O}^*(n^3)$ for one sample in one phase (Theorem~\ref{thm:mixing_LC}) times $\mathcal{O}^*(n)$ samples per phase for rounding (Section~\ref{sec:rounding}) times $K=\mathcal{O}^*(\sqrt{n})$ phases. The resulting distribution, however, is only $e\gamma$-close to the distribution with density proportional to $\exp\{-nF/\epsilon\}$ (by Theorem~\ref{Thm:Warm}). The guarantee of Theorem~\ref{eq:opt_guarantee} holds for the latter distribution, and we need to upper bound the effect of having a sample from an almost-desired distribution. Thankfully, $\gamma$ enters logarithmically in oracle complexity. Since $f$ is $L$-Lipschitz and domain is bounded, the range of function values over $\K$ is bounded by $B=\mathcal{O}(nLR)$. Then $\gamma$ can be chosen as $\epsilon/B$, which again only impacts oracle complexity by terms logarithmic in $n, L, R$.
\end{proof}

\end{appendix}
\bibliographystyle{apalike}
\bibliography{bibfile}

\end{document}

%% file: header.tex
\newtheorem{theorem}{Theorem}
\newtheorem{lemma}{Lemma}

\newtheorem{corollary}{Corollary}

\newtheorem{remark}{Remark}

\newtheorem{definition}{Definition}

\let\inf\undef
\DeclareMathOperator*{\inf}{\vphantom{p}inf}
\let\sup\undef
\DeclareMathOperator*{\sup}{\vphantom{p}sup}

\newcommand \red[1]{{\color[rgb]{0.80,0.00,0.00}#1}}

\newcommand{\mrm}[1]{\mathrm{#1}}

\newcommand{\K}{\mathcal{K}}
\newcommand{\B}{\mathcal{B}}
\newcommand{\oracle}{\boldsymbol{\mathrm{O}}}
\newcommand{\net}{\boldsymbol{\mathcal{N}}}
\newcommand{\E}{\mathbb{E}}

\newcommand{\argmin}[1]{\underset{#1}{\mrm{argmin}} \ }

\newcommand{\reals}{\mathbb{R}}
\newcommand{\En}{\mathbb{E}}  

\newcommand\Y{\mathcal{Y}}

\newcommand\F{\mathcal{F}}

